\documentclass[11pt]{article}

\usepackage{amssymb}
\usepackage{amsmath}
\usepackage{mathtools}
\usepackage{amsthm}
\usepackage{latexsym}
\usepackage{amsfonts}
\usepackage{amsmath,color}
\usepackage[T1]{fontenc}
\usepackage{url}
\usepackage{hyperref}

\theoremstyle{plain}
\newtheorem{theorem}{Theorem}[section]
\newtheorem{proposition}[theorem]{Proposition}
\newtheorem{lemma}[theorem]{Lemma}

\newtheorem{claim}[theorem]{Claim}
\newtheorem{conjecture}[theorem]{Conjecture}

\theoremstyle{definition}

\theoremstyle{remark}

\newcommand{\G}{\Gamma}

\usepackage{authblk}

\title{\bf A spectral characterization of the {\LARGE $s$}-clique extension of the square grid graphs}
\author{Sakander Hayat$^{1}$, Jack H. Koolen$^{1,2,}$\footnote{Corresponding author}~,
Muhammad Riaz$^{1}$}
\affil{$^{1}$School of Mathematical Sciences,\\
University of Science and Technology of China (USTC),\\
Hefei, Anhui, P.R. China\vspace{0.2cm}\\
$^{2}$Wen-Tsun Wu Key Laboratory of CAS,\\
School of Mathematical Sciences,\\
University of Science and Technology of China (USTC),\\
Hefei, Anhui, 230026, P.R. China\vspace{0.2cm}\\
Email addresses:  sakander@mail.ustc.edu.cn(SH), koolen@ustc.edu.cn(JHK), muhammadriaz84@gmail.com(MR)}

\date{}

\begin{document}

\maketitle

\vspace{-1cm}
\begin{abstract}
In this paper we show that for integers $s\geq2$, $t\geq1$,
any co-edge-regular graph which is cospectral with the $s$-clique extension of the $t\times t$-grid
is the $s$-clique extension of the $t\times t$-grid, if $t$ is large enough.
Gavrilyuk and Koolen used a weaker version of this result to show that the Grassmann graph $J_q(2D,D)$
is characterized by its intersection array as a distance-regular graph, if $D$ is large enough.
\end{abstract}
\centerline{{\bf 2010 Mathematics Subject Classification:} 05C50, 05C75, 05E30}
\begin{quote}
\end{quote}

\section{Introduction}\label{sec1}
For undefined terminologies, see the next section. For the definitions related to distance-regular graphs, see \cite{BCN89,DaKoTa}.\\

Recall that a regular graph is co-edge-regular, if there exists a constant $\mu$ such that
any two distinct and non-adjacent vertices have exactly $\mu$ common neighbors.
Our main result in this paper is as follows:
\begin{theorem}\label{main-result}
Let $\G$ be a co-edge-regular graph with spectrum
\begin{equation*}
\{\big(s(2t+1)-1\big)^{1},(st-1)^{2t},(-1)^{(s-1)(t+1)^{2}},(-s-1)^{t^2}\},
\end{equation*}
where $s\geq2$ and $t\geq1$ are integers. If $t\geq11(s+1)^3(s+2)$, then $\G$ is the $s$-clique extension
of the $(t+1)\times(t+1)$-grid.
\end{theorem}
\vspace{0.2cm}

For $s=2$, Yang et al. \cite{YAK17} gave a similar characterization by using Hoffman graphs.
The advantage of their method is that they do not need co-edge-regularity, but their method
seems difficult to generalize to graphs with smallest eigenvalue less than $-3$.\\

In \cite{GK2018}, Gavrilyuk \& Koolen showed that any local graph, say $\Delta$,
of a distance-regular graph $\G$ with the same intersection array as a $J_q(2D,D)$ is co-edge-regular
with constant $\mu=2q$ and it has the same spectrum as the $q$-clique extension of a
$\frac{q^{D}-1}{q-1}\times \frac{q^{D}-1}{q-1}$-grid, if $D\geq4$. So our main result shows
that any local graph $\Delta$ of $\G$ is really the $q$-clique extension of a
$\frac{q^{D}-1}{q-1}\times \frac{q^{D}-1}{q-1}$-grid, if $D$ is large enough. Gavrilyuk \& Koolen \cite{GK2018}
used the $Q$-polynomial property of the distance-regular graph $\G$ to conclude that any local graph of $\G$ is the
$q$-clique extension of a $\frac{q^{D}-1}{q-1}\times \frac{q^{D}-1}{q-1}$-grid.
Using this fact, they showed that the distance-regular graph $\G$
has to be the Grassmann graph $J_q(2D,D)$, if $D$ is large enough.\\

Another motivation comes from the lecture notes \cite{Terwilliger}. In these notes, Terwilliger
shows that any local graph of a thin $Q$-polynomial distance-regular graph is co-edge-regular
and has at most five distinct eigenvalues. So it is interesting to study co-edge-regular graphs
with a few distinct eigenvalues.\\

A further motivation is as follows.
Strongly regular graphs have attracted a lot of attention, see for example
the survey papers \cite{BV84,Seidel79}. On the other hand, there are only a few papers
on regular graphs with four distinct eigenvalues, see for example \cite{vD95,vD99,VS98,HH17}.
We think that connected co-edge-regular graphs with four distinct
eigenvalues is an interesting class of graphs as they are related to association schemes with
three classes. We believe that some strong results can be obtained for this class
of graphs. For example, we believe the following conjecture is true.

\vspace{0.1cm}
\begin{conjecture}\label{conj1}
Let $\G$ be a connected co-edge-regular graph with four distinct eigenvalues. Let
$t\geq2$ be an integer and $|V(\G)|=n(\G)$. Then there exists a constant $n_{t}$ such that, if $\theta_{\min}(\G)\geq -t$
and $n(\G)\geq n_t$ both hold, then $\G$ is the $s$-clique extension of a strongly regular graph, for some $2\leq s\leq t-1$.
\end{conjecture}
\vspace{0.2cm}

Note that for any integer $s\geq2$, the $s$-clique
extension of a connected non-complete strongly regular graph is co-edge-regular and has four distinct
eigenvalues. The main result of this paper can be seen as a first step towards Conjecture \ref{conj1}.

\section{Preliminaries}

\subsection{Definitions}
Let $G$ be a simple connected graph on vertex set $V(G)$ with $n=|V(G)|$, edge set $E(G)$ and adjacency matrix $A$.
The eigenvalues of $G$ are the eigenvalues of $A$. Let $\theta_{0},\theta_1,\ldots,\theta_t$
be the distinct eigenvalues of $G$ and $m_i$ be the multiplicity of $\theta_i$ ($i=0,1,\ldots,t$).
Then the multiset $\{\theta_0^{m_0},\theta_1^{m_1},\ldots,\theta_t^{m_t}\}$ is called the \emph{spectrum} of $G$.
Two graphs are called \emph{cospectral}, if they have the same spectrum.\\

For two distinct vertices $x$ and $y$, we write $x\sim y$ (resp. $x\nsim y$)
if they are adjacent (resp. nonadjacent) to each other. For a vertex $x$ in $G$, we define
$N_{G}(x)=\{y\in V(G)\mid y\sim x\}$, and $N_{G}(x)$ is called the \emph{neighborhood} of $x$.
The graph induced by $N_{G}(x)$ is called the \emph{local graph} of $G$ with respect to $x$ and is denoted by $G(x)$.
We call $d_{x}:=|N_{G}(x)|$, the \emph{valency} of $x$.
We denote the number of common neighbors between two distinct vertices $x$ and $y$
by $\lambda_{x,y}$ (resp. $\mu_{x,y}$) if $x\sim y$  (resp. $x\nsim y$).\\

A regular graph with $n$ vertices and valency $k$ is called \emph{co-edge-regular} with parameters
$(n,k,\mu)$, if any two nonadjacent vertices have precisely $\mu=\mu(G)$ common neighbors.
Moreover, a co-edge-regular graph is called \emph{strongly regular} with parameters $(n,k,\lambda,\mu)$
if, in addition, any two adjacent vertices have precisely $\lambda=\lambda(G)$ common neighbors.
A graph $G$ is called \emph{walk-regular}, if for all $r$, the number of closed walks of length $r$
from a given vertex $x$ is independent of the choice of $x$. Since this number equals $A^{r}_{xx}$, it
is the same as saying that $A^{r}$ has constant diagonal for all $r$.\\

Recall that a \emph{clique} (or a complete graph) is a graph
in which every pair of vertices is adjacent. A $t$\emph{-clique} is a clique with $t$ vertices
and is denoted by $K_{t}$, where $t$ is a positive integer.
The $t\times t$\emph{-grid} is the Cartesian product $K_{t}\Box K_{t}$.
Note that the $t\times t$-grid is strongly regular with parameters $(t^2,2t-2,t-2,2)$.
The spectrum of the $t\times t$-grid is $\big\{(2t-2)^1,(t-2)^{2(t-1)},(-2)^{(t-1)^{2}}\big\}$.\\

The \emph{Kronecker product} $M_1\otimes M_2$ of two matrices $M_1$ and $M_2$
is obtained by replacing the $ij$-entry of $M_1$ by $(M_1)_{i,j}M_{2}$ for all $i$ and $j$.
Note that if $\tau$ and $\eta$ are eigenvalues of $M_1$ and $M_2$ respectively, then $\tau\eta$
is an eigenvalue of $M_1\otimes M_2$. We refer to \cite[Section 9.7]{GD01} for a detailed treatment on
the Kronecker product of matrices and graph eigenvalues.


\subsection{Clique extensions of the square grid graphs}\label{subsec2.3}
In this subsection, we define $s$-clique extensions of graphs and we will give some specific results for the
the $s$-clique extension of the square grid graphs.\\

For a positive integer $s$, the \emph{$s$-clique extension} of $G$ is the graph $\widetilde{G}$ obtained from
$G$ by replacing each vertex $x\in V(G)$ by a clique $\widetilde{X}$ with $s$ vertices, such that $\tilde{x}\sim\tilde{y}$
(for $\tilde{x}\in \widetilde{X},~\tilde{y}\in \widetilde{Y}$) in $\widetilde{G}$ if and only if
$x\sim y$ in $G$. If $\widetilde{G}$ is the $s$-clique extension of $G$, then $\widetilde{G}$ has adjacency matrix $J_{s}\otimes(A+I_n)-I_{sn}$,
where $J_s$ is the all-ones matrix of size $s$ and $I_n$ is the identity matrix of size $n$. In particular, if $G$ has spectrum
\begin{equation}\label{G-spectrum}
\{\theta_0^{m_0},\theta_1^{m_1},\ldots,\theta_t^{m_t}\},
\end{equation}
then it follows that the spectrum of $\widetilde{G}$ is
\begin{equation}\label{sclique-spectrum}
\big\{\big(s(\theta_0+1)-1\big)^{m_0}, \big(s(\theta_1+1)-1\big)^{m_1},\ldots,\big(s(\theta_t+1)-1\big)^{m_t},(-1)^{(s-1)(m_0+m_1+\ldots+m_t)}\big\}.
\end{equation}
\\
\vspace{-0.6cm}

In case that $G$ is a connected regular graph with valency $k$ and with adjacency matrix $A$ having exactly four distinct eigenvalues
$\{\theta_0=k,\theta_1,\theta_2,\theta_3\}$, then $A$ satisfies the following (see for example \cite{H63}):
\begin{equation}\label{eq-hoffman-poly}
A^{3}-\bigg(\sum\limits_{i=1}^{3}\theta_i\bigg)A^{2}+\bigg(\sum\limits_{1\leq i<j\leq3}\theta_i\theta_j\bigg)A-\theta_1\theta_2\theta_3I=
\frac{\prod_{i=1}^{3}(k-\theta_i)}{n}J.
\end{equation}
This implies that $G$ is walk-regular, as was also shown in \cite{vD95}.\\

Now we assume that $\G$ is a cospectral graph with the $s$-clique extension of $(t+1)\times(t+1)$-grid, where $s\geq2$ is an integer.
Then by (\ref{G-spectrum}) and (\ref{sclique-spectrum}), the graph $\G$ has spectrum
\begin{equation}\label{sclique-grid-spectrum}
\big\{\theta_{0}^{m_0},\theta_{1}^{m_1},\theta_{2}^{m_2},\theta_{3}^{m_3}\big\}=
\big\{\big(s(2t+1)-1\big)^{1},(st-1)^{2t},(-1)^{(s-1)(t+1)^{2}},(-s-1)^{t^2}\big\}.
\end{equation}
Using (\ref{eq-hoffman-poly}) we obtain
\begin{equation*}
A^{3}+(3+s-st)A^{2}+(3+2s-s^2t-2st)A+(1+s-s^2t-st)I=2s^2(2t+1)J.
\end{equation*}
Thus, we have
\begin{equation}\label{walk-regularity}
A^{3}_{(x,y)}=\left\{
  \begin{array}{ll}
    2s^2t^2+4s^2t-6st+s^2-3s+2, & \hbox{\textrm{if}~$x=y$;} \\
    5s^2t+2st+2s^2-2s-3-(3+s-st)\lambda_{x,y}, & \hbox{\textrm{if}~$x\sim y$;} \\
    4s^2t+2s^2-(3+s-st)\mu_{x,y}, & \hbox{\textrm{if}~$x\nsim y$.}
  \end{array}
\right.
\end{equation}
\\

Applying the Hoffman bound \cite[Lemma 9.6.2]{GD01} to the complement of $\G$, we obtain:
\begin{lemma}\label{lem2}
Let $\G$ be a graph that is cospectral with the $s$-clique extension of the $(t+1)\times(t+1)$-grid, where $s\geq2$, $t\geq1$ are integers.
For any clique $C$ of $\G$, say with order $c$, we have $c\leq st+s$.
If equality holds, then every vertex $x\in V(\G)-V(C)$ has exactly $s$ neighbors in $C$.
\end{lemma}

\section{Lines in $\G$}\label{sec3}

At the end of Subsection \ref{subsec2.3}, we assumed that $\G$ is a graph cospectral with the $s$-clique extension of the $(t+1)\times(t+1)$-grid, where $s \geq 2$ and $t \geq 1$ are integers.
This implies that $\G$ is walk-regular.
For this section and the next three sections, we assume that $\G$ is co-edge-regular as well, i.e. there exists
exactly $\mu$ number of common neighbors between any two nonadjacent vertices of $\G$.
Note that, from the spectrum of the $s$-clique extension of the $(t+1)\times (t+1)$-grid, we obtain that $\mu=2s$.\\

Let $\G(\infty)$ be the local graph of $\G$ at vertex $\infty$. Assume that the vertices of $\G(\infty)$ have valencies
$d_1,\ldots,d_k$, where $k=s(2t+1)-1$. Then, as $\G$ is walk-regular and co-edge-regular, we obtain that
the number of walks of length two inside $\G(\infty)$ is same as in the local graph of the $s$-clique extension of
the $(t+1)\times(t+1)$-grid. Using (\ref{walk-regularity}), this implies that the sum of valencies and
the sum of square of valencies of vertices in $\G(\infty)$ are constant, and are given by the following equations, where $\varepsilon$ is the number of edges inside $\G(\infty)$.
\begin{equation}\label{local-valencies-sum}
2\varepsilon =\sum\limits_{i=1}^{k}d_i= 2st(st+2s-3)+s^{2}-3s+2,
\end{equation}
\begin{equation}\label{local-valencies-square-sum}
\sum\limits_{i=1}^{k}(d_i)^{2}= 2st(s^2t^2+4s^2t-6st+3s^2-10s+8)+s^3-5s^2+8s-4.
\end{equation}
By (\ref{local-valencies-sum}) and (\ref{local-valencies-square-sum}), we obtain the following equation.
\begin{equation}\label{primal-eqn}
\sum\limits_{i=1}^{k}\big(d_i-(st+s-2)\big)^{2}=s^2t^2(s-1).
\end{equation}
It turns out that (\ref{primal-eqn}) is of crucial importance in proving our main result.\\

It is straightforward to show the following lemma. We will use it later.
\begin{lemma}\label{integers-lemma}
Let $a,~b,~p$ and $q$ be some integers satisfying $0\leq q\leq p\leq a$ and $b\leq a$. If $p+q=a+b$, then $p^2+q^2\leq a^2+b^2$.
\end{lemma}

We call a maximal clique a line, if it contains at least
$\frac{3}{4}s(t+2)$ vertices. The following result shows the existence of lines in $\G$.
\begin{proposition}\label{prop1}
Let $\G$ be a co-edge-regular graph, that is cospectral with the $s$-clique extension of the $(t+1)\times(t+1)$-grid, where $s \geq 2$ and $t \geq 1$ are integers.
If $t\geq 11(s+1)^3(s+2)$, then any vertex in $\G$ lies on exactly two lines.
\end{proposition}
\begin{proof}
Note that $\G$ is regular with valency $k=s(2t+1)-1$, and co-edge-regular with constant $\mu=2s$.
Let $\infty$ be a vertex of $\G$ and $\G(\infty)$ be its local graph. As before let $\varepsilon$ be the number of edges inside $\G(\infty)$.
We first obtain a lower bound on $\varepsilon$.
By (\ref{local-valencies-sum}) we obtain
\begin{equation}\label{ineq-bounds}
2\varepsilon \geq 2s^2t^2,
\end{equation}as $s \geq 2$.
Now we derive an upper bound on $\varepsilon$.
Let $P$
be a coclique with maximum order in $\G(\infty)$ with vertex set $\{x_1,x_2,\ldots,x_p\}$.
Then by the interlacing theorem \cite[Theorem 9.5.1]{GD01}, we have $p\leq (s+1)^{2}$, as $\theta_{\min}=-s-1$,
where $\theta_{\min}$ is the smallest eigenvalue of $\G$.
We define
\begin{center}
$B:=\{y\sim\infty\mid$ $y$ has at least two neighbors in $P\}$.
\end{center}
Let $b$ be the cardinality of $B$. Then
\begin{equation}\label{b-eqn}
b\leq (2s-1){p\choose 2}\leq sp(p-1)\leq s^{2}(s+1)^{2}(s+2),
\end{equation}
as $\mu=2s$. We also define the sets $A_i$ such that
\begin{center}
$A_i:=\{y\sim\infty\mid$ $y$ has only $x_i$ as its neighbor in $P\}\cup\{x_{i}\}$,
\end{center}
where $a_i:=|A_i|$. Note that, if $\mathcal{A}:=\bigcup\limits_{i}A_i$, then $\mathcal{A}\bigcup B=V\big(\G(\infty)\big)$.
Note further that any two vertices in $A_i$ are adjacent, because $P$ is maximum.
Thus, for every $i$, $A_i$ is a clique in $\G(\infty)$.
\\

Now, inside $\mathcal{A}$, we have at most $\frac{1}{2}\big(\sum\limits_{i}a_i(a_{i}-1)+(p-1)(2s-1)\sum\limits_{i}a_i\big)$ edges,
inside $B$, we have at most $\frac{1}{2}b(b-1)$ edges and between $\mathcal{A}$ and $B$, we have at most $b(k-b)$ edges.
Then, by counting these number of edges in $\G(\infty)$, we obtain
\begin{eqnarray}\label{epsilon}
\nonumber 2\varepsilon&\leq&\sum\limits_{i}a_i(a_{i}-1)+b(b-1)+2(k-b)b+(p-1)(2s-1)\sum\limits_{i}a_{i}\\
\nonumber &\leq&\sum\limits_{i}a_i(a_{i}-1)+2(k-1)b+(p-1)(2s-1)\sum\limits_{i}a_{i}\\
\nonumber &\leq&\sum\limits_{i}a_i(a_{i}-1)+2(k-1)s^{2}(s+1)^{2}(s+2)+k(s^2+2s)(2s-1)~~~~~~~~~~~~\big(\textrm{as}~~\sum\limits_{i}a_{i}\leq k\big)\\
\nonumber &\leq&\sum\limits_{i}a_i(a_{i}-1)+4s^{3}(s+1)^{2}(s+2)(t+1)+4s^2(s+1)^2(t+1)~~~\big(\textrm{as}~~k=s(2t+1)-1\leq2s(t+1)\big)\\
&\leq&\sum\limits_{i}a_i(a_{i}-1)+4s^2(s+1)^3(s+2)(t+1).
\end{eqnarray}
Without loss of generality, we may assume that $a_1\geq a_2\geq\ldots\geq a_p\geq1$. By Lemma \ref{lem2},
we have $a_1< s(t+1)$. Assume that $\G$ has at most one line through $\infty$. Then
$a_1<s(t+1)$, $a_2\leq \frac{3}{4}s(t+2)$ and $\sum\limits_{i}a_{i}\leq k < 2s(t+1).$ Using Lemma \ref{integers-lemma}, we obtain
\begin{equation}\label{upper-bound-ai}
\sum\limits_{i}a_i(a_{i}-1)< \sum\limits_{i}a_i^2< \big(s(t+1)\big)^2+\bigg(\frac{3}{4}s(t+2)\bigg)^2+\bigg(\frac{1}{4}s(t+2)\bigg)^2 <\frac{13}{8}\big(s(t+2)\big)^2 .
\end{equation}
Combining (\ref{ineq-bounds}), (\ref{epsilon}) and (\ref{upper-bound-ai}), we obtain
\begin{equation}\label{new}
2s^2t^2 \leq 2 \varepsilon \leq \frac{13}{8}\big(s(t+2)\big)^2 +4s^2(s+1)^3(s+2)(t+1).
\end{equation}
As $$t \geq11(s+1)^3(s+2)> 600, \ \ \ \ (s \geq 2)$$
we find $t+2<\frac{201}{200}t$. Putting this in (\ref{new}), we obtain a contradiction.\\

\noindent Thus, if $t\geq11(s+1)^3(s+2)$, we obtain that every vertex
in $\G$ lies on at least two lines. Let $C_1$ and $C_2$ be these two lines and $m:=| V(C_1)\bigcap V(C_2)|$. Since $\G$
is co-edge-regular with $\mu=2s$, we obtain that $m\leq 2s$. As $3\times \frac{3}{4}s(t+2)-6s > 2s(t+1) > k$, we find that every vertex
in $\G$ lies on at most two lines. This shows the proposition.
\end{proof}

Now we prove the following property for lines through a vertex.
\begin{lemma}\label{l+m=s-1}
Let $\G$ be a co-edge-regular graph, that is cospectral with the $s$-clique extension of the $(t+1)\times(t+1)$-grid,
where $s \geq 2$ and $t \geq 1$ are integers.
Let $\infty$ be a vertex of $\G$. Let $C_1$ and $C_2$ be the two lines through $\infty$. Let $\ell:=|\{y\mid y\notin V(C_1)\bigcup V(C_2) \}|$
and $m:=| V(C_1)\bigcap V(C_2)|$. If $t\geq11(s+1)^3(s+2)$, then $\ell+m=s$, and both $C_1$ and $C_2$ have at least $s(t-1)+1$ vertices and at most
$s(t+1)$ vertices.
\end{lemma}
\begin{proof}
Let $\G(\infty)$ be the local graph corresponding to vertex $\infty$ of $\G$, with vertex set $\{ x_1, x_2, \ldots , x_k\}$. For $1\leq i\leq k$, let $d_{i}$ be the valency of vertex in $\G(\infty)$.
Without loss of generality, we may assume that $\{ x_1, x_2, \ldots, x_{\ell}\}$ are the vertices in $\G$ which do not lie in either $C_1$ or $C_2$.
Note that for $x_i\notin V(C_1)\bigcup V(C_2)$, we have $d_{i}\leq (6s-3)+\frac{1}{2}s(t+1)$. This, in turn,  implies, using
(\ref{primal-eqn})

\begin{eqnarray}\label{Eqz2}
\nonumber s^{2}t^{2}(s-1) =\sum\limits_{i=1}^{k}\big(d_{i}-(st+s-2)\big)^2&\geq&\sum\limits_{y\notin V(C_1)\bigcup V(C_2)}\big(d_{y}-(st+s-2)\big)^2\\
\nonumber&\geq&\ell\bigg(\frac{1}{2}s(t+1)-6s\bigg)^2\\
&=&\ell\bigg(\frac{1}{4}\big(s(t-11)\big)^2\bigg).
\end{eqnarray}
Hence we find that
\begin{equation}\label{claimh1}
\ell\leq 5s,
\end{equation}
as $t\geq11(s+1)^{3}(s+2)>100$.\\

By (\ref{claimh1}), it follows that
\begin{eqnarray}\label{jack1}
\nonumber s(t+1)\geq|V(C_j)|&\geq& 1+ k - s(t+1)-  \ell \\
\nonumber &\geq& 2s(t+1) -s(t+1) - 5 s\\
&=&st-5s
\end{eqnarray}
for $j=1,2$.

Now we are going to refine the estimates for the number $\varepsilon$  of edges in $\G(\infty)$. This will give the following claim.

\begin{claim}\label{claim2}
$\ell+m=s$.
\end{claim}
\noindent{\bf Proof of Claim \ref{claim2}.}
Let $X:=V(C_1)\bigcap V(C_2)-\{\infty\}$, $Y:=V(C_1)\triangle V(C_2)$, and $Z:=\G(\infty)-V(C_1)\bigcup V(C_2)$. For $u \in X \bigcup Y \bigcup Z$, we will write $d_u$ for the valency of $u$ inside $\G(\infty)$. So we need to estimate $|X| + |Z| = \ell + m -1$.

\noindent For $x\in X$, we have
\begin{equation*}\label{sakka2}
2s(t+1)-5s\leq d_{x}\leq k-1=s(2t+1)-2.
\end{equation*}
\noindent Now, for $y\in Y$, (\ref{jack1}) and $\mu = 2s$ give us
\begin{equation*}\label{sakka4}
     st - 5s -1 \leq   d_y \leq s(t+1) -2 + 2s-1.
\end{equation*}
\noindent Moreover, for $z \in Z$, we have
\begin{equation*}\label{sakka1}
d_{z}\leq 2s-1+2s-1+\ell-1\leq 9s-2.
\end{equation*}
\noindent So, for $u\in X\bigcup Z$, we have
\begin{equation*}\label{sakka3}
st-8s\leq d_{u}-st- s+2 \leq st +1,
\end{equation*}
and for $y \in Y$ we have
\begin{equation*}\label{sakka}
|d_y - st-s +2| \leq 6s.
\end{equation*}

\noindent By using the above expressions and (\ref{primal-eqn}), we obtain
\begin{eqnarray}\label{sakka5}
\nonumber (\ell+m-1)\big(st-8s\big)^2 &\leq& \sum\limits_{v \in X \cup Z}\big(d_{v}-(st+s-2)\big)^{2}\\
\nonumber &\leq & \sum\limits_{v\sim \infty}\big(d_{v}-(st+s-2)\big)^{2}=s^2t^2(s-1)\\
&\leq &(\ell+m-1)(st+s)^2+2st(6s)^2.
\end{eqnarray}
By simplifying (\ref{sakka5}), we obtain
\begin{equation}\label{sakka6}
\frac{t^2(s-1)-72st}{(t+1)^2}\leq\ell+m-1\leq\frac{s^2t^2(s-1)}{\big(s(t-8)\big)^2}.
\end{equation}
If $t\geq11(s+1)^3(s+2)> 200$, then (\ref{sakka6}) implies that
\begin{equation*}
s-2<\ell+m-1 <s.
\end{equation*}
This shows that $\ell+m=s$.
\qed
\\

As, by Proposition \ref{prop1} we have that every vertex lies in two lines, we obtain using Claim \ref{claim2} that
for a line $C$ in $\G$ with order $c$, we have $s(t-1)+1\leq c\leq s(t+1)$, if $t\geq11(s+1)^3(s+2)$.
This shows the lemma.
\end{proof}

As a consequence of Lemma \ref{l+m=s-1}, we show the following result.
\begin{lemma}\label{c1+c2=2m+2st}
Let $\G$ be a co-edge-regular graph, that is cospectral with the $s$-clique extension of the $(t+1)\times(t+1)$-grid, where $s\geq2$ and $t\geq1$ are integers.
Let $C_1$ and $C_2$ be two lines with respective order $c_1$ and $c_2$. Assume that $C_1$ and $C_2$ intersect in exactly
$m$ vertices, where $m\geq1$. If $t\geq11(s+1)^3(s+2)$, then $c_{1}+c_{2}=2st+2m$.
\end{lemma}
\begin{proof}
Let $x\in V(C_1)\bigcap V(C_2)$. Let $\ell=|\G(x)-V(C_1)-V(C_2)|$. Then we have
\begin{equation*}
(c_1-m)+(c_2-m)+m-1+\ell=k=s(2t+1)-1.
\end{equation*}
If $t\geq11(s+1)^3(s+2)$, then, by Lemma \ref{l+m=s-1},  we have $\ell+m=s$. Hence we obtain
\begin{equation*}
c_{1}+c_{2}=2st+2m.
\end{equation*}
\end{proof}

\section{The order of lines}\label{sec4}

In this section we will show the following lemma on the order of lines.
\begin{lemma}\label{sum-of-cliques-pairs} Let $s \geq 2$ and $t \geq 1$ be integers.
Let $\G$ be a co-edge-regular graph, that is cospectral with the $s$-clique extension of the $(t+1)\times(t+1)$-grid.
 Let $q_{i}$ be the number of lines of order $s(t-1)+i$ in $\G$ where $i=1,\ldots,2s$ and $\delta = \sum\limits_{i=1}^{2s} q_i$ be the number of lines in $\Gamma$.
Assume $t\geq11(s+1)^3(s+2)$. Then
\begin{equation}\label{main-eqn}
\sum\limits_{i=1}^{2s}\big(s(t-1)+i\big)q_{i}=2s(t+1)^{2}
\end{equation} holds, and the number $\delta$ satisfies
\begin{equation}\label{total-cliques-ineq}
2t+2\leq\delta\leq2t+6,
\end{equation}
where $\delta=2t+2$ implies that $q_{i}=0$ for all $i<2s$ and $q_{2s}=2t+2$.\\

\noindent Let $\alpha := \delta - 2t -2$. Then $\alpha \in \{0, 1, 2, 3, 4\}$ and
\begin{equation}\label{eqz3}
\sum\limits_{i=1}^{2s}(2s-i)q_i =\alpha s(t+1)
\end{equation}
holds.
\end{lemma}
\begin{proof}
Assume $t\geq11(s+1)^3(s+2)$.
By Lemma \ref{l+m=s-1}, any vertex lies on exactly two lines and each line has
at least $st-s+1$ and at most $st+s$ vertices.
\noindent Now consider the set
\begin{equation*}
W=\{(x,C)\mid x\in V(C),~\mathrm{where}~C~\mathrm{is~a~line}\}.
\end{equation*}
Then, by double counting the cardinality of the set $W$, we obtain (\ref{main-eqn}).\\

\noindent By (\ref{main-eqn}), we obtain that
\begin{equation*}
\delta=\sum\limits_{i=1}^{2s}q_i<\sum\limits_{i=1}^{2s}\frac{s(t-1)+i}{s(t-1)}q_{i}=\frac{2(t+1)^2}{t-1}=2(t+3)+\frac{8}{t-1}.
\end{equation*}
Thus, if $t\geq10$, then we obtain
\begin{equation*}
\delta\leq2t+6.
\end{equation*}
In a similar fashion, we obtain
\begin{equation*}
\delta\geq\sum\limits_{i=1}^{2s}\frac{s(t-1)+i}{s(t+1)}q_{i}=2(t+1).
\end{equation*}
This shows that  $\delta\geq2t+2$  and $\delta=2t+2$
implies that all lines have order $st+s$.\\

Let $\alpha:=\delta-2t-2$. Then $\alpha\in\{0,1,2,3,4\}$. From (\ref{main-eqn}), we find
\begin{equation*}
\sum\limits_{i=1}^{2s}(2s-i)q_i =\sum\limits_{i=1}^{2s} s(t+1)q_{i} -\sum\limits_{i=1}^{2s}\big(s(t-1)+i\big)q_{i} =\big(\delta-2(t+1)\big)s(t+1) = \alpha s(t+1).
\end{equation*}
This shows (\ref{eqz3}).
This completes the proof.
\end{proof}

\section{The neighborhood of a line}
In this section we will show the following proposition.
\begin{proposition}\label{propneigh}
Let $\G$ be a graph that is cospectral with the $s$-clique extension of the $(t+1)\times(t+1)$-grid, where $s \geq 2$ and $t \geq 1$ are integers.
If  $t\geq11(s+1)^3(s+2)$, then $\G$ contains exactly $2t+2$ lines.
\end{proposition}
\begin{proof}
In Lemma \ref{sum-of-cliques-pairs}, we have seen that the number $\delta$ of lines satisfies $2t+2 \leq \delta \leq 2t +6$.
Now we will assume that $\delta \geq 2t+3$, in order to obtain a contradiction.  We will show this contradiction in a number of claims.
In order to obtain this contradiction, we will look at the neighborhood of a line in $\G$.\\

Note that the condition $t\geq11(s+1)^3(s+2)$ implies $t \geq 5s \geq 10$, which we will use several times below.
As before, let $q_{i}$ be the number of lines of order $s(t-1)+i$ in $\G$, where $i=1,\ldots,2s$.
Let $h$ be minimal such that $q_{h}\neq0$. Fix a line $C$ with exactly $s(t-1)+h$ vertices.
Let $q'_i:=q'_{i}(C)$ be the number of lines $C'$ with $s(t-1)+i$ vertices that intersect $C$
in at least one vertex. Let $\tau$ be the number of lines that intersect $C$ in at least one vertex. Note that by Lemma \ref{c1+c2=2m+2st}, we obtain that
\begin{equation}\label{primal1}
|V(C)\bigcap V(C')|=\frac{h+i-2s}{2}.
\end{equation}
This means that $q'_i =0 $ if $i < \max\{ h, 2s+2 - h\}$.
We first give the following claim concerning the numbers $\tau$, $q_{i}$ and $q'_{i}$ ($i=h,\ldots,2s$).
\begin{claim}\label{claim1}
The following hold:
\begin{itemize}
   \item[\emph{(i)}] $q'_i =0 $ if $i < \max\{h, 2s+2 - h\}$,
   \item[\emph{(ii)}] $q'_{i}\leq q_{i}$ if $i>h$,
  \item[\emph{(iii)}] $q'_{i}\leq q_{i}-1$ if $i=h$,
  \item[\emph{(iv)}] $\tau=\sum\limits_{i=h}^{2s}q'_{i}\leq\bigg(\sum\limits_{i=h}^{2s}q_{i}\bigg)-1\leq2t+5$,
  \item[\emph{(v)}] $\sum\limits_{i=\max\{h, 2s+2 - h\}}^{2s}q'_{i}\bigg(\frac{h+i-2s}{2}\bigg)=s(t-1)+h$.
\end{itemize}
\end{claim}
\noindent{\bf Proof of Claim \ref{claim1}}.
(i)-(iii) follow easily from the definitions of $q'_i$, $h$ and  $(\ref{primal1})$.
(iv) follows immediately from (i)-(iii).
To show (v), recall that every vertex of $\G$ lies in exactly two lines so this, in particular,
holds for the vertices of $C$. By using (\ref{primal1}) and (i), we obtain (v).
\qed
\\

Next we show that $h\geq s$.
\begin{claim}\label{myclaim0}
$h\geq s$.
\end{claim}
\noindent{\bf Proof of Claim \ref{myclaim0}.}
By Claim \ref{claim1}(ii), (iii) and (v), we have
\begin{eqnarray*}
s(t-1)+h&=&\sum\limits_{i=\max\{h, 2s+2 - h\}}^{2s}q'_{i}\bigg(\frac{h+i-2s}{2}\bigg)\\
&\leq&\sum\limits_{i=h}^{2s}q'_{i}\bigg(\frac{h}{2}\bigg)\\
&<&\sum\limits_{i=h}^{2s}q_{i}\bigg(\frac{h}{2}\bigg)=\frac{h}{2}\delta.\end{eqnarray*}
By $\delta\leq2t+6$, we obtain
\begin{equation}\label{myeqn0}
(t+2)h = (2t+4)\frac{h}{2}\geq s(t-1).
\end{equation}
Since we have $t\geq5s \geq 10$, (\ref{myeqn0}) implies that $h>s-\frac{1}{2}$. This shows the claim.
\qed
\\

Note that Claim \ref{myclaim0} shows that $\frac{h+i-2s}{2} \geq 0$ for all $i \geq h$. Using this, we show that $h \geq \frac{3}{2}s$.
\begin{claim}\label{myclaim1}
$h\geq \lceil\frac{3}{2}s\rceil$.
\end{claim}
\noindent{\bf Proof of Claim \ref{myclaim1}.}
By Claim \ref{claim1} we have
\begin{equation}\label{myeqn1}
\sum\limits_{i=h}^{2s}q_{i}\bigg(\frac{h+i-2s}{2}\bigg)\geq\sum\limits_{i=h}^{2s}q'_{i}\bigg(\frac{h+i-2s}{2}\bigg)=\sum\limits_{i=\max\{h, 2s+2 - h\}}^{2s}q'_{i}\bigg(\frac{h+i-2s}{2}\bigg)=s(t-1)+h.
\end{equation}
Adding twice (\ref{myeqn1}) to (\ref{eqz3}), we obtain
\begin{equation*}
h\delta=h\sum\limits_{i=h}^{2s}q_{i}\geq 2s(t-1)+2h+\alpha s(t+1).
\end{equation*}
As $2t+3 \leq \delta\leq2t+6$, we have
\begin{equation}\label{myeqn3}
h(2t+4)\geq 2s(t-1)+\alpha s(t+1)\geq3st-s.
\end{equation}
This implies that
\begin{equation*}
h\geq \frac{3st-s}{2t+4}=s\bigg(\frac{3t-1}{2t+4}\bigg).
\end{equation*}
Since $t\geq5s\geq 10$, we obtain that $h\geq \lceil\frac{3}{2}s\rceil$.
This shows the claim.
\qed
\\

The following claim gives an upper bound on $h$.
\begin{claim}\label{myclaim2}
$h\leq\lfloor\frac{3}{2}s\rfloor$.
\end{claim}
\noindent{\bf Proof of Claim \ref{myclaim2}.}
As $ \alpha \geq 1$ and  $\delta\leq2t+6$,   we obtain, by (\ref{eqz3}), the following
\begin{equation*}
(2s-h)(2t+6)\geq(2s-h)\sum\limits_{i=h}^{2s}q_{i}\geq s(t+1).
\end{equation*}
This implies that
\begin{equation*}
2s-h\geq\frac{s(t+1)}{2t+6}=\frac{1}{2}s-\frac{2s}{2t+6}.
\end{equation*}
Since $t\geq5s$, it follows that $h\leq\lfloor\frac{3}{2}s\rfloor$.
This shows the claim.
\qed
\\

Claims \ref{myclaim0}, \ref{myclaim1} and \ref{myclaim2} imply the following claim.
\begin{claim}\label{myclaim3}
$h=\frac{3}{2}s$, $\alpha=1$ and $\delta = 2t+3$.
\end{claim}
\noindent{\bf Proof of Claim \ref{myclaim3}.}
By Claims \ref{myclaim1} and \ref{myclaim2}, we obtain $h=\frac{3}{2}s$.
As $h=\frac{3}{2}s$, (\ref{myeqn3}) becomes
\begin{equation*}
\frac{3}{2}s(2t+4)\geq2s(t-1)+\alpha s(t+1).
\end{equation*}
Now, by $ t \geq  5s \geq 10$, we obtain  $\alpha<2$, which implies $\alpha=1$.
This also shows that $\delta = 2t+3.$ This shows the claim.
\qed
\\

Next we show that all lines have order $st + \frac{s}{2}$ and the number $\tau$ of lines that intersect $C$ is equal to  $2t+1$.
\begin{claim}\label{myclaim4}
We have  $\tau = 2t+1$.
\end{claim}
\noindent{\bf Proof of Claim \ref{myclaim4}.}
By Claim \ref{myclaim3} we have $h = \frac{3}{2} s$ and $\delta = 2t+3$. As $\alpha =1$, Eq.  (\ref{eqz3}) gives us
\begin{equation}\label{myeqn5}
\sum\limits_{i=h}^{2s}q_{i}(2s-i)= s(t+1).
\end{equation}
Inserting $h=\frac{3}{2}s$ in Claim \ref{claim1} (v), we obtain
\begin{eqnarray}\label{myeqn6}
\nonumber \sum\limits_{i=h}^{2s}q'_{i}(h-s) &\leq& \sum\limits_{i=h}^{2s}q'_{i}\bigg(\frac{h+i-2s}{2}\bigg)\\
\nonumber &=& s(t-1)+\frac{3}{2}s,\\
&=& s\big(t+\frac{1}{2}\big)
\end{eqnarray}

\noindent
As $\tau=\sum\limits_{i=h}^{2s}q'_{i}$ and $h = \frac{3}{2}s$, we find
\begin{equation}\label{jack}
\frac{1}{2}\tau s= \tau(h-s) \geq s\big(t+\frac{1}{2}\big),
\end{equation}
which implies that
\begin{equation}\label{myeqn7}
\tau\leq2t+1,
\end{equation}
holds.

On the other side, by Claim \ref{claim1} (v) and $h = \frac{3}{2} s$, we have
\begin{equation}\label{myeqn8}
\sum\limits_{i=h}^{2s}q_{i}\bigg(\frac{h+i-2s}{2}\bigg)\geq\sum\limits_{i=h}^{2s}q'_{i}\bigg(\frac{h+i-2s}{2}\bigg)+(\delta-\tau)(h-s)= s\big(t+\frac{1}{2}\big)+(\delta-\tau)\frac{1}{2}s.
\end{equation}
By adding twice (\ref{myeqn8}) to (\ref{myeqn5}), we obtain
\begin{equation}\label{myeqn9}
h\delta=h\sum\limits_{i=h}^{2s}q_{i}\geq2s\big(t+\frac{1}{2}\big)+s(t+1)+(\delta-\tau)s,
\end{equation}
which implies
\begin{equation}\label{myeqn10}
2t+1 - \tau = \delta-\tau\leq\frac{5}{2},
\end{equation}
as $\delta = 2t+3$.
Combining (\ref{myeqn10}) and (\ref{myeqn7}), we find
\begin{equation*}
\tau=2t+1.
\end{equation*}
This shows the claim.
\qed
\\

\noindent As we have equality in (\ref{myeqn7}), and consequently,  in (\ref{jack}) and in (\ref{myeqn6}), we obtain
\begin{equation*}
q'_{i}=0,~(i>h)~~\textrm{and}~~q'_{h}=2t+1.
\end{equation*}

In order to finish the proof of the proposition, we construct a graph $\G'$ whose vertices are the lines in $\G$ where
two lines are adjacent if they intersect in at least one vertex in $\G$.
Let $C'$ be the line that does not intersect $C$. As every vertex of $\G$ lies in two lines
this implies that $C'$ has a common neighbor $C''$ with $C$ in $\G'$. Now replacing $C$
by $C''$, we see that $C'$ has also $s(t-1)+h$ vertices. This means all lines of $\G$ have exactly $s(t-1)+h$ vertices. Hence $\G'$ is a $(2t+1)$-regular graph on $2t+3$ vertices and this  is clearly not possible.
So we obtain a contradiction and this finishes the proof of the proposition.
\end{proof}

\section{Proof of the main result}

In this section we show our main result, namely Theorem \ref{main-result}.\\

\noindent{\bf Proof of Theorem \ref{main-result}.}
Assume $t\geq11(s+1)^3(s+2)$. By Propositions \ref{prop1} and \ref{propneigh} and Lemma \ref{sum-of-cliques-pairs}, we find that there are exactly $2t+2$ lines, each of order $s(t+1)$, and
every vertex $x$ in $\G$ lies on exactly two lines.
Moreover, by Lemma \ref{c1+c2=2m+2st}, the two lines through any vertex $x$ have exactly $s$ vertices in common, and every neighbor of $x$ lies in one of the two lines through $x$.
Now, consider the following equivalence relation
$\mathcal{R}$ on the vertex set $V(\G)$:
\begin{center}
$x\mathcal{R}x'$ if and only if $\{x\}\cup N_{\G}(x)=\{x'\}\cup N_{\G}(x')$, where $x,x'\in V(\G)$.
\end{center}

Every equivalence class under $\mathcal{R}$ contains $s$ vertices and it is the intersection of two lines.
Let us define the graph $\widehat{\G}$ whose vertices are the equivalent classes and two classes, say $S_1$ and $S_2$, are adjacent in $\widehat{\G}$
if and only if any vertex in $S_1$ is adjacent to any vertex in $S_2$. Then $\widehat{\G}$ is a regular graph with valency $2t$,
and $\G$ is the $s$-clique extension of $\widehat{\G}$. Note that the spectrum of $\widehat{\G}$ is equal to
\begin{equation*}
\{(2t)^{1},(t-1)^{2t},(-2)^{t^{2}}\},
\end{equation*}
by the relation of the spectra of $\G$ and $\widehat{\G}$, see
(\ref{G-spectrum}) and (\ref{sclique-spectrum}).
Since $\widehat{\G}$ is a connected regular graph with valency $2t$, and since it has exactly three
distinct eigenvalues, it follows that $\widehat{\G}$ is a strongly regular graph with parameters $\big((t+1)^{2},2t,t-1,2\big)$.\\

From \cite{S59}, it follows that a graph with these parameters is the $(t+1)\times(t+1)$-grid or $t=3$ and the graph is the Shrikhande graph.
But the Shrikhande graph is not possible as $\widehat{\G}$ has cliques of order $t+1$. This completes the proof.
\qed

\section*{Acknowledgements}
Sakander Hayat is supported by a CAS-TWAS president's fellowship at USTC China.
Jack Koolen is supported by the National Natural Science Foundation of China (Nos. 11471009 and 11671376).


\begin{thebibliography}{9}

\bibitem{BCN89}
A.E. Brouwer, A.M. Cohen, A. Neumaier, {\it Distance-Regular
Graphs}, Springer-Verlag, New York, 1989.


\bibitem{BV84}
A.E. Brouwer, J.H. van Lint, Strongly regular graphs and partial geometries,
In {\em Enumeration and Design: Papers from the conference on combinatorics held
at the University of Waterloo, Waterloo, Ont., June 14-July 2, 1982}
(Ed. D. M. Jackson and S. A. Vanstone). Toronto, Canada: Academic Press, pp. 85-122, 1984.

\bibitem{vD95}
E.R. van Dam, Regular graphs with four eigenvalues,
{\em Linear Algebra Appl.}, 226--228 (1995), 139--162.

\bibitem{vD99}
E.R. van Dam, Three-class association schemes,
{\em J. Algebr. Combin.}, 10 (1999), 69--107.


\bibitem{VS98}
E. R. van Dam, E. Spence,
Small regular graphs with four eigenvalues,
{\em Discrete Math.}, 189 (1998), 233--257.

\bibitem{DaKoTa}
E.R. van Dam, J.H. Koolen, H. Tanaka, Distance-regular graphs,
{\em Electron. J. Combin.}, (2016), \#DS22.


\bibitem{GK2018}
A.L. Gavrilyuk, J.H. Koolen, On a characterization of Grassmann graphs,
In preparation.

\bibitem{GD01}
C.D. Godsil, G. Royle,
{\em Algebraic Graph Theory}, Springer-Verlag, Berlin, 2001.


\bibitem{H63}
A.J. Hoffman, On the polynomial of a graph,
{\em Amer. Math. Monthly}, 70 (1963), 30--36.

\bibitem{HH17}
X. Huang, Q. Huang, On regular graphs with four distinct eigenvalues,
{\em Linear Algebra Appl.}, 512 (2017), 219--233.




\bibitem{Seidel79}
J.J. Seidel, ``Strongly regular graphs'', In {\em Surveys in Combinatorics (Proc. Seventh British Combinatorial Conf., Cambridge, 1979)},
Cambridge, England: Cambridge University Press, pp. 157-180, 1979.

\bibitem{S59}
S.S. Shrikhande, The uniqueness of $L_{2}$ association schemes,
{\em Ann. Math. Statist.}, 30 (1959), 781--798.

\bibitem{Terwilliger}
P. Terwilliger,
{\em Algebraic Graph Theory}, Lecture notes, Unpublished, See the link https://icu-hsuzuki.github.io/lecturenote/.


\bibitem{YAK17}
Q. Yang, A. Abiad, J.H. Koolen, An application of Hoffman graphs for spectral characterizations of graphs,
{\em Electron. J. Combin.}, 24(1) (2017), \#P1.12.

\end{thebibliography}
\end{document}